\documentclass[12pt,oneside]{amsart}
\usepackage{hyperref}
\usepackage{graphicx}
\usepackage{amscd}
\usepackage{amsmath}
\usepackage{amsfonts}
\usepackage{bbm}
\usepackage{pdfsync}

\usepackage{amssymb}
\usepackage{verbatim}
\usepackage{comment}

\usepackage{color}

\usepackage{subfigure}

\newcommand*{\ind}{\mathbbm{1}}
\input prepictex
\input pictexwd
\input postpictex

\newtheorem{theorem}{Theorem}
\theoremstyle{plain}

\newtheorem{fact}[theorem]{Fact}

\newtheorem{definition}[theorem]{Definition}

\newtheorem{lemma}[theorem]{Lemma}

\newtheorem{proposition}[theorem]{Proposition}

\numberwithin{equation}{section}

\newcommand{\ien}{\underline{i}_n}
\newcommand{\jen}{\underline{j}_n}

\newcommand{\R}{\mathbb R}

\begin{document}

\parindent0pt

\title[Fractal percolations]
{Projections of fractal percolations}

\author{Micha\l\ Rams}
\address{Micha\l\ Rams, Institute of Mathematics, Polish Academy of Sciences, ul. \'Sniadeckich 8, 00-956 Warsaw, Poland
\tt{rams@impan.gov.pl}}

\author{K\'{a}roly Simon}
\address{K\'{a}roly Simon, Institute of Mathematics, Technical
University of Budapest, H-1529 B.O.box 91, Hungary
\tt{simonk@math.bme.hu}}

 \thanks{
\indent The research of Rams was supported by the EU FP6 Marie
Curie program CODY and by the Polish MNiSW Grant NN201 0222 33
�Chaos, fraktale i dynamika konforemna�.
 The research of Simon was supported by OTKA Foundation
\#71693}

\begin{abstract}{In this paper we study the radial and orthogonal projections and the distance sets of the random Cantor sets $E\subset
\mathbb{R}^2 $ which are  called Mandelbrot percolation or
percolation fractals. We prove that the following assertion holds
almost surely: if the Hausdorff dimension of $E$ is greater than
$1$ then the orthogonal projection to \textbf{every} line,
the radial projection with \textbf{every} center, and distance set from \textbf{every} point contain intervals.}
\end{abstract}

\maketitle



\section{Introduction}

\begin{figure}[ht!]
  \includegraphics[width=8cm]{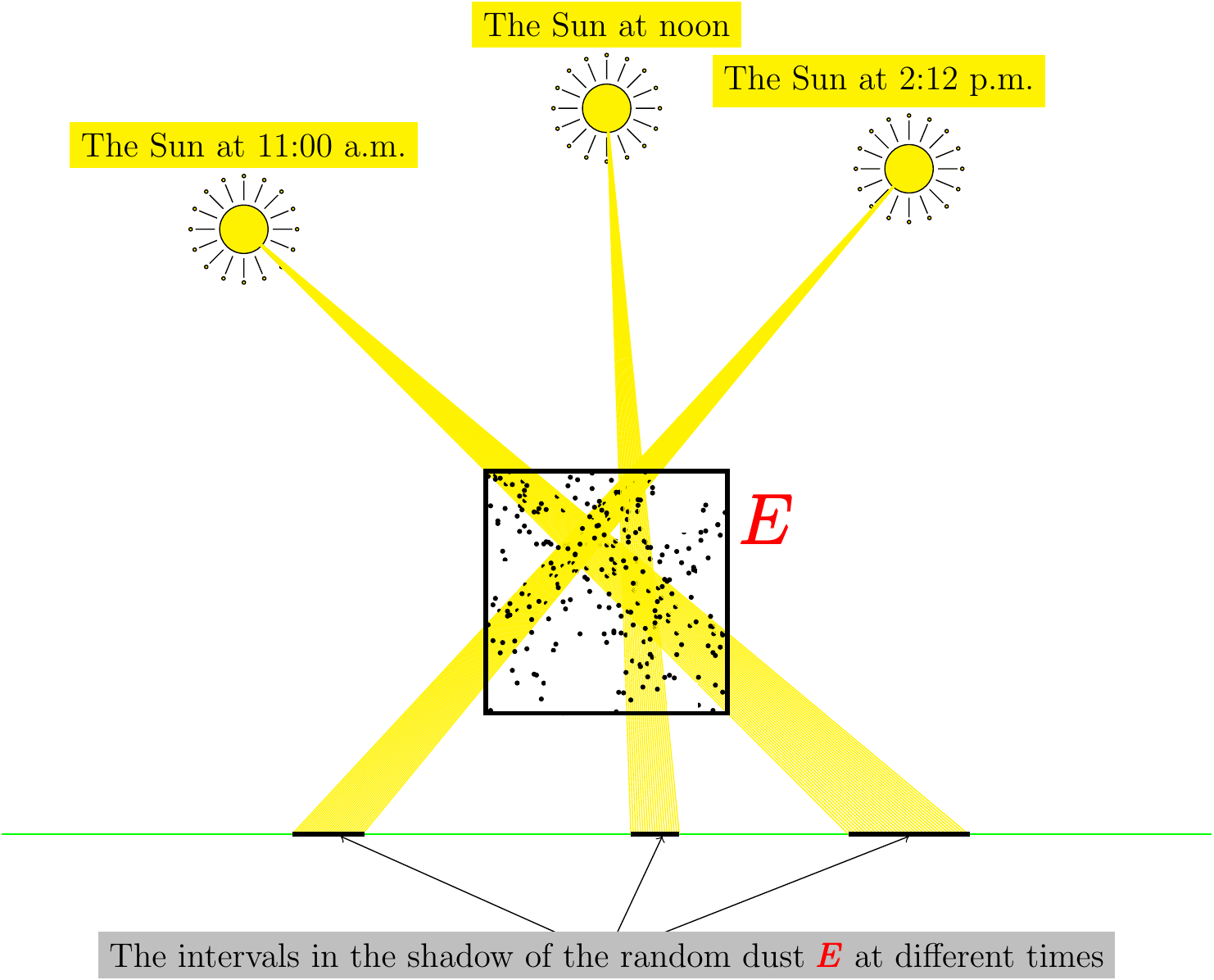}\\
\end{figure}

This picture shows what we prove: although the fractal percolation is a Cantor dust, it throws a thick shadow \textbf{at any time}. Here thick means containing at least one non-trivial open interval (we will write simply 'containing intervals'). One does not need to rotate it to use it as an umbrella.

\bigskip

In order to construct a model for turbulence Mandelbrot introduced
\cite{M} a random set which is now called Mandelbrot percolation
or fractal percolations or canonical curdling. In the simplest
case (we consider a more general case in this paper), we are given
a natural number $M\geq 2$ and a probability $p\in (0,1)$. First
we partition the unit square $[0,1]^2$ into $M^2$ congruent
squares and then we retain each of them with probability $p$ and
discard them with probability $1-p$ independently. In the squares
which were retained we repeat this process independently ad
infinitum. The random set $E\subset [0,1]^2$ that results is the
fractal percolation or canonical curdling. In fact in this paper
sometimes we consider the more general setup where the $M^2$
congruent squares, mentioned above, are chosen with not
necessarily the same probabilities.

\medskip

These random Cantor sets have attracted considerable attention. In
1978 Peyri\`ere
computed the almost sure Hausdorff dimension, conditioned on
non-extinction (this result was reproved many times). In 1988   Chayes, Chayes, Durett \cite{CCD}
proved that there is a critical probability $p_c$ such for every
$0<p<p_c$ the random Cantor set $E$ is totally disconnected, but
for every $p>p_c$ with positive probability $E$ percolates. This
means that there is a connected component in $E$ which connects
the left hand side wall to the right hand side wall of the unit
square $[0,1]^2$ with positive probability. Dekking and Meester \cite{DM} gave a simplified
proof for the previously mentioned result and defined several
phases such that as we increase $p$ the process passes through all
of these phases.
If the fractal is totally disconnected  ($p<p_c$) it still can happen that some of its projections contain intervals.

The orthogonal projections of fractals on the plane were already studied by Marstrand \cite{Ma}
 in 1954.
 Marstrand's Theorem
 says that for any set $A\subset \mathbb{R}^2$  with $\dim_{\rm H}A>1$ the orthogonal projection of $A$ to almost all lines has  positive Lebesgue measure; here $\dim_{\rm H}$ denotes the Hausdorff dimension.

Existence of an interval in the orthogonal projections of some Cantor sets in the plane was first studied in relation with the famous Palis conjecture about the algebraic difference of  Cantor sets.
The algebraic difference of the Cantor sets  $C_1,C_2$ is the $45^{\circ}$
projection of $C_1\times C_2$. Palis conjectured that "typically" $C_1-C_2$ is either small in the sense that it has Lebesgue measure zero or big in the sense that it contains some intervals.
See e.g.  \cite{PT}, \cite{MY},  \cite{PS}, \cite{DekSim}.

For percolation fractals Falconer and Grimmett \cite{FG} studied the existence of intervals in the vertical and horizontal projections. Our work is a generalization of their result.

\medskip

\textbf{The most important conclusion} of our result is that
whenever the probability $p>1/M$  then  although the set $E$ may
be totally disconnected, almost surely conditioned on
non-extinction, all projections in various families (orthogonal,
radial, co-radial projections) contain some intervals. On the other hand,
if $p\leq 1/M$ this cannot happen. Namely, Falconer \cite{FRF} proved that in this case the one dimensional Hausdorff measure of $E$ is almost surely zero.

\medskip

\textbf{The paper is organized as follows.} In the second section we give precise definitions of the objects we study, we also formulate our main results. In the third section we explain the importance of statistical self-similarity. The fourth section contains the proof of Theorem \ref{thm:ortho}. In the fifth section we add one more idea that lets us upgrade this argument, yielding the proof of Theorem \ref{thm:many}. Finally, in the sixth section we formulate the most general form of our results, Theorem \ref{thm:general}, and Theorem \ref{thm:nonlin} follows as a special case.

\section{Notation and results}

\subsection{Mandelbrot percolation}

First we provide a definition of the random Cantor set $E$ (we will call it the fractal percolation) which is the object of interest of this paper. Given
$$
M\geq 2\mbox{ and } p_{i,j}\in [0,1] \mbox{ for every }
i,j\in \left\{0,\dots ,M-1\right\},
$$
we partition the unit square $K=[0,1]^2$ into $M^2$ congruent squares of side length $1/M$.
\begin{equation*}\label{153}
  K=\bigcup _{i,j=0}^{M-1}K_{i,j}\mbox{ where }
 K_{i,j}:=\left[\frac{i}{M},\frac{i+1}{M}\right]\times
\left[\frac{j}{M},\frac{j+1}{M}\right].
\end{equation*}
In the first step, we retain the square $K_{i,j}$ with probability $p_{i,j}$ and we discard $K_{i,j}$ with
probability $1-p_{i,j}$ for every $(i,j)\in
\left\{0,\dots ,M-1\right\}^2$ independently. The union of squares
retained is denoted $E_1$. Within each square $K_{i,j}\subset E_1$
we repeat the process described above independently. The squares
of side length $1/M^2$ retained are called level two squares and
their union is called $E_2$. Similarly, for every $n$ we construct
the set $E_n$. The object of interest
in this paper is the random set $E:=\cap _{n=1}^{\infty }E_n$.

More formally, let $\mathcal{T}_n$ be the partition of $K$ into $M$-adic squares of level $n$. For each square $L\in \mathcal{T}_n$ we can find two sequences $\{i_1,\ldots,i_n\}, \{j_1,\ldots,j_n\}\in \{0,\ldots,M-1\}^n$ such that

\[
L=\left[\sum_{l=1}^n i_l\cdot M^{-l}, \sum_{l=1}^n i_l\cdot M^{-l} + M^{-n}\right] \times \left[\sum_{l=1}^n j_l\cdot M^{-l}, \sum_{l=1}^n j_l\cdot M^{-l} + M^{-n}\right].
\]
We will denote such square by $K_{\ien, \jen}$, where

\[
\underline{i}_n:=(i_1,\dots ,i_n),\ \underline{j}_n:=(j_1,\dots ,j_n)\in \left\{0,\dots ,M-1\right\}^n.
\]
Clearly, $K_{\underline{i}_{n+1}, \underline{j}_{n+1}} \subset K_{\ien', \jen'}$ if and only if

\[
i_k=i_k', j_k=j_k' \mbox{ for all } k=1,\ldots,n.
\]

We define $\mathcal{E}_0 = \mathcal{T}_0 = (\emptyset, \emptyset)$ and then we construct inductively a random family $\{\mathcal{E}_n\}, \mathcal{E}_n \subset \mathcal{T}_n$. That is, if $(\ien; \jen)\notin \mathcal{E}_n$ then $(i_1,\ldots,i_n,i; j_1,\ldots,j_n,j) \notin \mathcal{E}_{n+1}$ for all $i,j\in \{0,\ldots,M-1\}$ and if $(\ien; \jen)\in \mathcal{E}_n$ then $(i_1,\ldots,i_n,i; j_1,\ldots,j_n,j) \in \mathcal{E}_{n+1}$ with probability $p_{i,j}$. Those events are jointly independent.

We denote

\[
E_n = \bigcup_{(\ien;\jen)\in \mathcal{E}_n} K_{\ien,\jen}
\]
and

\[
E = \bigcap_{n=1}^\infty E_n.
\]
The sequence $\{E_n\}$ is a decreasing sequence of compact sets, hence $E$ is nonempty if and only if all $E_n$ are nonempty. It follows easily from the general theory of branching processes, see for example \cite[Theorem 1]{AtNe},  that
\begin{equation*}\label{163}
 E\ne \emptyset  \mbox{  with positive probability }
\mbox{ if and only if } \sum\limits_{0\leq i,j\leq M-1}p_{i,j}>1.
\end{equation*}

We will always assume
\begin{equation*}\label{162}
\sum\limits_{i,j=0}^{M-1}p_{i,j}>M
\end{equation*}
and our results will be conditioned on $E$ being nonempty.
It was proved by several authors:
Peyri\`{e}re \cite{Pe}, Hawkes, \cite{Ha}
Falconer \cite{FRF} and  Mauldin, Williams \cite{MWRF} and Graf \cite{Gr}
that
\begin{equation*}\label{155}
\mbox{ If } E\ne \emptyset \mbox{ then }
\dim_{\rm H}(E)=\frac{\log \left(\sum\limits_{i,j=0}^{M-1}p_{i,j}\right)}{\log M}\mbox{ a.s. }
\end{equation*}
In particular, under our assumptions $\dim_{\rm H} E >1$ (provided $E$ is nonempty).

The proof of this statement involves proving the following:

\begin{fact} \label{en}The following assertion holds almost surely:

If $E(\omega )\ne\emptyset $ then
\begin{equation*}\label{78}
  \lim\limits_{n\to\infty} \frac{1}{n}\log \#\mathcal{E}_n(\omega )\to \log\sum\limits_{i,j=0}^{M-1}
  p_{ij}.
\end{equation*}
\end{fact}

\subsection{Projections}

The object of our study is the existence of intervals in different
kinds of projections of $E$. The nature of projections of
angles $0$ or $\pi /2$ is conspicuously different and these cases
were already treated by Falconer and Grimmett in \cite{FG}. So, mostly we restrict our
attention to the domain of angles
$$
\mathfrak{D}:=\left(0,\pi /2\right)\cup \left(\pi /2,\pi \right).
$$

It will be convenient for us to use a special form of projections. Instead of the 'usual' orthogonal projection ${\rm proj}_\alpha$ onto some line  we will use projection $\Pi_\alpha$, the codomain of which is one of diagonals of $K$. If $\alpha\in (0,\pi/2)$ (i.e. if the projection is in upper left - lower right direction) we will use the nonorthogonal projection in direction $\alpha$ onto the interval $([0,0],[1,1])$. Otherwise, if $\alpha\in(\pi/2, \pi)$ and the projection is in the upper right - lower left direction, we will project onto the interval $([0,1],[1,0])$. Naturally, ${\rm proj}_\alpha(E)$ contains an interval if and only if $\Pi_\alpha(E)$ does. See Figure \ref{151}.

\begin{figure}
\begin{center}
\includegraphics[width=13.5cm]{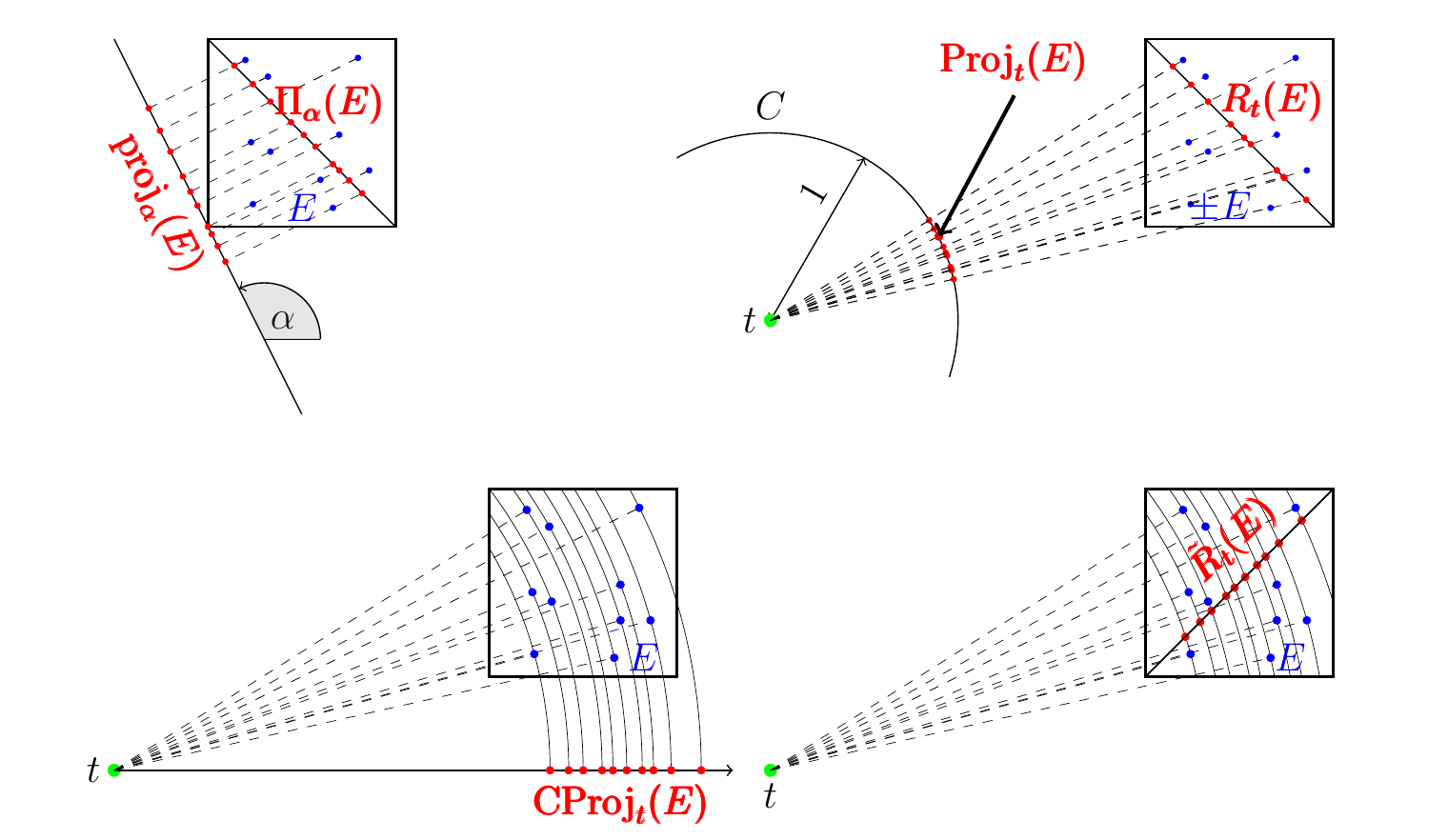}
\caption{The orthogonal $\mathrm{proj}_\alpha $, radial $\mathrm{Proj_t}$, co-radial $\mathrm{CProj}_t$ projections  and the auxiliary projections $\Pi _\alpha $, $R_t$, and $\tilde{R}_t$. }
\label{151}
\end{center}
\end{figure}


We are going to consider nonlinear projections of $E$ as well. Given $t\in \R^2$, the radial projection with center $t$ of set $E$ is denoted by ${\rm Proj}_t(E)$ and is defined as the set of angles under which points of $E\setminus\{t\}$ are visible from $t$. Given $t\in\R^2$, the co-radial projection with center $t$ of a set $E$ is denoted by ${\rm CProj}_t(E)$ and is defined as the set of distances between $t$ and points from $E$. Figure \ref{151} explains why we consider this object a projection.

Like in the case of orthogonal projections, we will consider auxiliary formulations. If the point $t$ is in 'diagonal' direction from $K$ (i.e. if both X and Y coordinates of $t$ are outside $[0,1]$) then instead of ${\rm Proj}_t$ with codomain $S^1$ and ${\rm CProj}_t$ with codomain $\R_+$ we can consider $R_t$ and $\tilde{R}_t$, whose codomains are diagonals of $K$. For example, as shown in Figure \ref{151}, if $t$ is in lower left direction from $K$ (both coordinates of $t$ are negative) then the codomain of $R_t$ is $([0,1],[1,0])$ and the codomain of $\tilde{R}_t$ is $([0,0],[1,1])$. Once again, $R_t(E)$ contains an interval if and only if ${\rm Proj}_t(E)$ does, and similarly for ${\rm CProj}_t(E)$ and $\tilde{R}_t(E)$.

There is a more general notion of a {\bf family of almost linear projections} we are going to use, but it is more complicated. The definition will be given in the last section.

\subsection{Results}

Let us start from a direct generalization of \cite{FG}. Let $\alpha\in \mathcal{D}$. In the fourth section we will define condition $A(\alpha)$ on the set of probabilities $\{p_{i,j}\}$, at the moment it is enough to know that if all $p_{i,j}>M^{-1}$ then $A(\alpha)$ is satisfied for all $\alpha\in\mathcal{D}$.

\begin{theorem} \label{thm:ortho}
Let $\alpha\in\mathcal{D}$. If $A(\alpha)$ holds and $E$ is nonempty then almost surely ${\rm proj}_\alpha(E)$ contains an interval.
\end{theorem}

Our next, stronger, result lets us consider projections in all directions at once.

\begin{theorem} \label{thm:many}
Assume that $A(\alpha)$ holds for all $\alpha\in \mathcal{D}$ and that $E\neq\emptyset$. To handle the horizontal and vertical projections, we also assume that
\[
\forall i,j\in \{0,\ldots,M-1\}, \ \sum_{l=0}^{M-1} p_{i,l} > 1 \mbox{ and } \sum_{k=0}^{M-1} p_{k,j} >1.
\]
Then almost surely ${\rm proj}_\alpha(E)$ contains an interval for all $\alpha\in S^1$.
\end{theorem}

Finally, let us consider nonlinear projections.

\begin{theorem} \label{thm:nonlin}
Assume that $A(\alpha)$ holds for all $\alpha\in \mathcal{D}$ and that $E\neq\emptyset$. Then almost surely both ${\rm Proj}_t(E)$ and ${\rm CProj}_t(E)$ contain an interval for all $t\in \R^2$.
\end{theorem}

\section{Statistical self-similarity}

The goal of this section is to explain two simple ideas, explaining why the statistical self-similarity of the construction of $E$ simplifies our task.

Let $\varphi_{\ien, \jen}$ be the natural contraction sending $K$ onto $K_{\underline{i}_n,\underline{j}_n}$. That is,
\begin{equation*}\label{1}
\varphi _{\underline{i}_n,\underline{j}_n}
(x,y)=\frac{1}{M^n}\cdot (x,y)+t_{\underline{i}_n,\underline{j}_n},
\end{equation*}
where $t_{\ien, \jen}$ is the lower left corner of $K_{\ien, \jen}$. Then by the {\bf statistical self-similarity} of $E$ we mean the following fact: for any $K_{\ien, \jen}\in \mathcal{T}_n$ the conditional distribution of $E\cap K_{\ien,\jen}$ conditioned on $(\ien; \jen)\in \mathcal{E}_n$ is the same as distribution of $\varphi_{\ien, \jen}(E)$.

The first idea, used already in \cite{FG}, is as follows. Let $E$ be a nonempty realization of the fractal percolation. Almost surely, $E$ has infinitely many points, hence we can find an infinite sequence of numbers $n_k$ and squares $K_{\underline{i}_{n_k}, \underline{j}_{n_k}} \subset E_{n_k}$ such that any two squares $K_{\underline{i}_{n_k}, \underline{j}_{n_k}}$ are not contained in each other. Fix $\alpha$. The probability that ${\rm proj}_\alpha(E\cap K_{\underline{i}_{n_k}, \underline{j}_{n_k}})$ contains an interval is the same for each $k$ (and the same as probability that ${\rm proj}_\alpha(E)$ contains an interval) and those are independent events. Hence, it is enough to prove that ${\rm proj}_\alpha(E)$ contains an interval with positive probability to know that it contains an interval with probability 1 (conditioned on $E$ being nonempty).

The second idea is quite similar. Let $t\in \R^2$ and consider the radial projection with center $t$ (for co-radial projection it works much the same). Once again, if $E$ is nonempty then we can almost surely find a square $K_{\underline{i}_{l}, \underline{j}_{l}}$ with nonempty intersection with $E$ and which $t$ does not belong to and is in diagonal direction from ($E$ almost surely is not contained in a horizontal or vertical line). We can then construct the family $K_{\underline{i}_{n_k}, \underline{j}_{n_k}}$ of subsets of $K_{\underline{i}_{l}, \underline{j}_{l}}$ such that the size of each $K_{\underline{i}_{n_k}, \underline{j}_{n_k}}$ is very small compared to its distance from $t$ (we just need to take them sufficiently small). Note that not only $t$ is in diagonal direction from each $K_{\underline{i}_{n_k}, \underline{j}_{n_k}}$, the direction is actually bounded away from horizontal and vertical.

The probability that ${\rm Proj}_t(E\cap K_{\underline{i}_{n_k}, \underline{j}_{n_k}})$ contains an interval is the same as probability that ${\rm Proj}_{\varphi_{\underline{i}_{n_k}, \underline{j}_{n_k}}^{-1}(t)}(E)$ contains an interval. Hence, to prove that ${\rm Proj}_t(E)$ almost surely contains an interval, it is enough to prove that the probability that ${\rm Proj}_{t'}(E)$ contains an interval is uniformly bounded away from zero for $t'$ far away from $K$ and in direction bounded away from horizontal and vertical.

It is a natural observation that the radial/co-radial projections with center sufficiently far away do not differ much from linear projections. Indeed, this is how this idea will be used in the proof of Theorem \ref{thm:nonlin} in the last section.

\section{Proof of Theorem \ref{thm:ortho}}

\subsection{Ideas}

As our main idea comes from paper of Falconer and Grimmett \cite{FG}, let us start by recalling their proof.
We will assume the simplest case: all the probabilities are equal to $p>M^{-1}$. We want to prove that the probability that the vertical projection of the percolation fractal is the whole interval $[0,1]$ is positive. For any $n>0$ let us divide $[0,1]$ into intervals of length $M^{-n}$ and let us code them by the usual $M$-adic codes. Over each interval $C(i_1,\ldots,i_n)$ there is a whole column of $M^n$ $M$-adic squares of level $n$, and $C(i_1,\ldots,i_n)$ is contained in the vertical projection of $E_n$ if and only if at least one of those squares belongs to $\mathcal{E}_n$. Denoting by $A_n(i_1,\ldots,i_n)$ the number of squares above $C(i_1,\ldots,i_n)$ contained in $\mathcal{E}_n$, we need to prove that with positive probability all $A_n(i_1,\ldots,i_n)$ (for all possible sequences $\ien$) are positive.

Note that
\begin{equation} \label{abcd}
{\mathbb E}(A_{n+1}(i_1,\ldots,i_n,j)| A_n(i_1,\ldots,i_n)=a) = Mpa.
\end{equation}

Choose any $\gamma \in (1, Mp)$ and let $G_n(i_1,\ldots,i_n)$ be the event that

\[
A_{n+1}(i_1,\ldots,i_n,j) > \gamma A_n(i_1,\ldots,i_n)
\]
for all $j=0,1,\ldots, M-1$. By large deviation estimations,

\[
1-P(G_n(i_1,\ldots,i_n)) \approx \tau^{A_n(i_1,\ldots,i_n)}
\]
for some $\tau <1$. Hence, if all the events $G_1(i_1),\ldots,G_{n-1}(i_1,\ldots,i_{n-1})$ hold then $A_n(i_1,\ldots,i_n) \geq \gamma^n$ and so

\[
P(G_n(i_1,\ldots,i_n)| G_1(i_1)\wedge\ldots\wedge G_{n-1}(i_1,\ldots,i_{n-1})) > 1-c \tau^{\gamma^n}.
\]
Hence, at level $n$ we have to check exponentially big number of events (precisely, $M^n$ of them) but each of those events is superexponentially certain to happen. It follows that with positive probability all those events will happen.

Our goal in this section is a more complicated statement: for the same kind of percolation fractal we fix a direction $\alpha$ (neither horizontal nor vertical) and we want to check that the projection of the fractal in this direction contains an interval with positive probability. Equation \eqref{abcd} does not hold: even if a point belongs to a projection of some $n$-th level square, it does not imply that the expected number of $n+1$-st level squares in the approximation of the percolation fractal such that their projections contain the point is greater than 1. More precisely, if the point belongs to the projection of the 'central' part of the square then everything might work, but not for the points very close to the ends of the projection interval.

To go around this technical problem, we only count the number of 'central' parts of projections of $n$-th level squares that a given point belongs to. This lets us replace \eqref{abcd} by Condition A($\alpha$) as our main working tool. Note that if we check that a sufficiently dense set of points belongs to 'central' parts of projections of some squares from $n$-th approximation of the fractal, the whole projections will cover everything. We only need to take care that the number of points needed at step $n$ grows at most exponentially fast with $n$ and then the Falconer and Grimmett's argument will go through.

\subsection{Condition A}

Time to give the details. We fix $\alpha\in\mathcal{D}$. We are going to consider $\Pi_\alpha$ instead of ${\rm proj}_\alpha$, i.e. we are projecting onto a diagonal $\Delta_\alpha$ of $K$. For any $(\ien, \jen)$ the map $\Pi_\alpha\circ \varphi_{\ien, \jen}: \Delta_\alpha \to \Delta_\alpha$ is a linear contraction of ratio $M^{-n}$. We will use its inverse: a map $\psi_{\alpha, \ien, \jen}: \Pi_\alpha(K_{\ien,\jen})\to \Delta_\alpha$. It is a linear expanding map (of ratio $M^n$) and it is onto.

Consider the class of nonnegative real functions on $\Delta_\alpha$, vanishing on the endpoints. There is a natural random inverse Markov operator $G_\alpha$ defined as

\[
G_{\alpha}f(x) = \sum_{(i,j)\in \mathcal{E}_1; x\in \Pi_\alpha(K_{i,j})} f\circ \psi_{\alpha, i, j}(x).
\]
The corresponding operator on the  $n$-th level is
$$
G^{(n)}_{\alpha}f(x)=
\sum\limits_{(\underline{i}_n,\underline{j}_n)\in \mathcal{E}_n; x\in \Pi_\alpha(K_{\ien,\jen})} f\circ\psi _{\alpha ,\underline{i}_n,\underline{j}_n}(x).
$$
In particular for any $H\subset \Delta^\alpha $ we have
\begin{equation*}\label{70}
 G_{\alpha}^{(n)}\ind_H(x)=
 \#\left\{(\underline{i}_n,\underline{j}_n)\in \mathcal{E}_n:
 x\in \Pi _\alpha \left(\varphi _{\underline{i}_n,\underline{j}_n}(H)\right)\right\}.
\end{equation*}
Although $G^{(n)}_\alpha$ should not be thought of as the $n$-th iterate of $G_\alpha$, the expected value of $G^{(n)}_\alpha$
is the $n$-th iterate of the expected value of $G_\alpha$. Namely, let
\begin{equation*}\label{92}
F_\alpha=\mathbb{E}\left[G_{\alpha}\right] \mbox{ and } F_\alpha^n=\mathbb{E}\left[G_{\alpha}^n\right]
\end{equation*}
We then have the formulas

\[
F_\alpha f(x) = \sum_{i,j; x\in \Pi_\alpha(K_{i,j})} p_{i,j}\cdot f\circ \psi_{\alpha, i, j}(x)
\]
and

\begin{equation*}\label{3}
F^n_\alpha f(x)=\sum\limits_{
(\underline{i}_n,\underline{j}_n); x\in \Pi_\alpha(K_{\ien,\jen})}
p_{\underline{i}_n,\underline{j}_n}\cdot f\circ \psi _{\alpha ,\underline{i}_n,\underline{j}_n}(x),
\end{equation*}

where

\[
p_{\ien, \jen} = \prod_{k=1}^n p_{i_k, j_k}.
\]
Hence, $F_\alpha^n$ is indeed the $n$-th iteration of $F_\alpha$ (which explains why we are allowed to use this notation).

\begin{definition} \label{def:conda}
We say the percolation model satisfies {\bf Condition A($\alpha$)} if there exist closed intervals $I_1^\alpha, I_2^\alpha\subset \Delta_\alpha$ and a positive integer $r_\alpha$ such that
\begin{itemize}
\item[i)] $I_1^\alpha\subset {\rm int}I_2^\alpha, I_2^\alpha \subset {\rm int}\Delta_\alpha$,
\item[ii)] $F_\alpha^{r_\alpha}\ind_{I_1^\alpha} \geq 2 \ind_{I_2^\alpha}$.
\end{itemize}
\end{definition}

It will be convenient to use additional notation. For $x\in\Delta_\alpha$, $\alpha\in \mathcal{D}$, and $I\subset \Delta_\alpha$ we denote

\[
D_n(x,I, \alpha) = \{(\ien, \jen); x\in \Pi_\alpha\circ \varphi_{\ien, \jen}(I)\}
.
\]
That is, if we write $\ell^\alpha (x)$ for the line segment through $x\in \Delta_\alpha  $ in direction $\alpha $, $D_n(x,I, \alpha)$ is the set $(\underline{i}_n,\underline{j}_n)$
for which $\ell^\alpha (x)$ intersects $\varphi_{\ien, \jen}(I)$.

The point ii) of Definition \ref{def:conda} can then be written as

\[
\forall_{x\in I_2^\alpha} \ \sum_{(\underline{i}_{r_\alpha}, \underline{j}_{r_\alpha}) \in D_{r_\alpha}(x, I^\alpha _1, \alpha)} p_{\underline{i}_{r_\alpha}, \underline{j}_{r_\alpha}} \geq 2.
\]

\begin{figure}[ht!]
  \includegraphics[width=8cm]{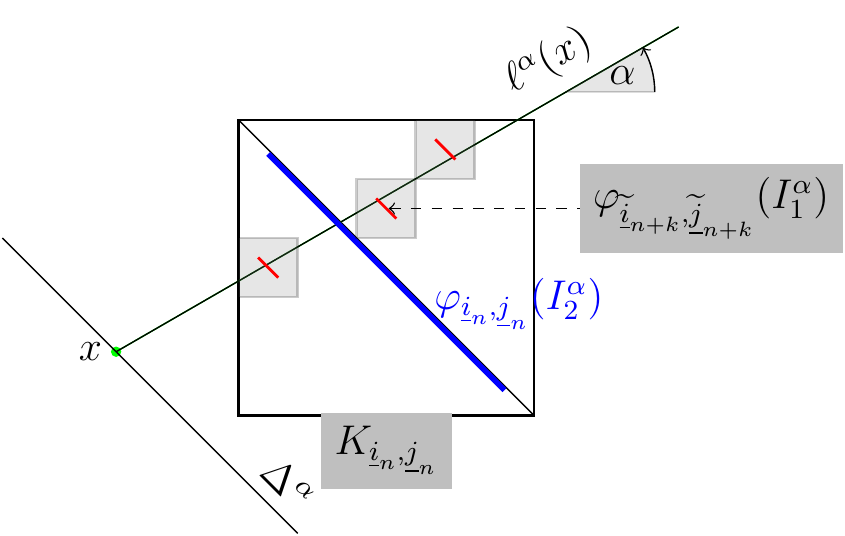}\\

 \caption{Condition $A(\alpha )$}\label{217}
\end{figure}
The heuristic explanation of Condition $A(\alpha )$ is as follows:
 If $(\underline{i}_n,\underline{j}_n)\in \mathcal{E}_n\cap D_n(x,I^\alpha _2,\alpha )$
then the expected number of
 $(\widetilde{\underline{i}}_{n+r},\widetilde{\underline{j}}_{n+r})$ such that

$K_{\widetilde{\underline{i}}_{n+r},\widetilde{\underline{j}}_{n+r}}\subset K_{\underline{i}_n,\underline{j}_n}$ and $(\widetilde{\underline{i}}_{n+r},\widetilde{\underline{j}}_{
n+r})\in \mathcal{E}_{n+r}\cap D_{n+r}(x,I^\alpha _1,\alpha )$ is at least $2
$. See Figure \ref{217}.

\subsection{Robustness}

In this subsection we will explain a very simple geometric idea we will use constantly in the last three sections.

Consider two parallel lines $l_1$, $l_2$. On $l_1$ we have an interval $I$. Let $J$ be the image of $I$ under linear projection onto $l_2$ in direction $\theta$. Let $I'$ be a greater interval on $l_1$, containing $I$ together with some neighbourhood. Then not only the projection of $I'$ onto $l_2$ in direction $\theta$ will contain $J$, but also if we perturb $\theta$ sufficiently slightly, the resulting projection of $I'$ will still contain $J$. Applying to our situation, whenever $(\ien, \jen)\in D_n(x, I_1^\alpha , \alpha )$  we will have $(\ien, \jen)\in D_n(y, I^\alpha _2, \beta )$ for all $y$ sufficiently close to $x$ and $\beta $ sufficiently close to $\alpha $.

First application: robustness of Condition A($\alpha$).

\begin{figure}[ht!]
  \includegraphics[width=8cm]{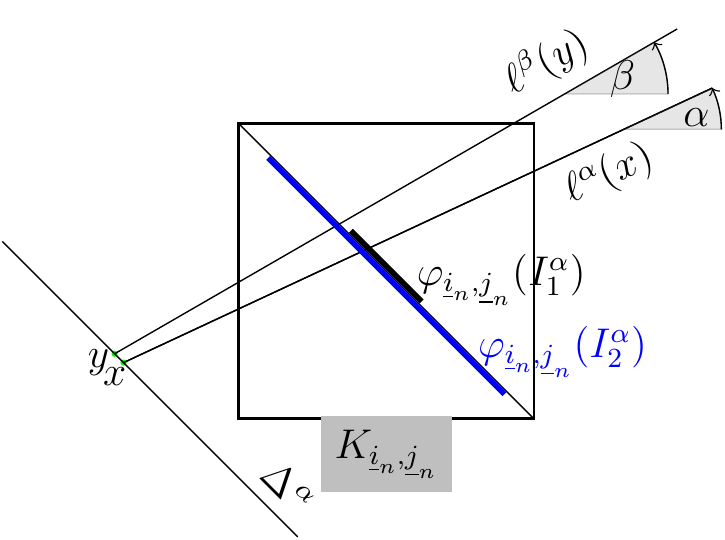}\\
 \caption{Robustness of Condition $A(\alpha )$}\label{218}
\end{figure}

\begin{proposition} \label{prop:arob}
If condition A($\alpha$) holds for some $\alpha\in\mathcal{D}$ for some $I_1^\alpha, I_2^\alpha$ and $r_\alpha$ then it will also hold in some neighbourhood $J\ni\alpha$. Moreover, for all $\theta\in J$ we can choose $I_1^\theta=I_1', I_2^\theta=I_2, r_\theta=r_\alpha$ not depending on $\theta$.
\end{proposition}
\begin{proof}
Let $\delta$ be the Hausdorff distance between $I_1^\alpha$ and $I_2^\alpha$, i.e. the greatest number for which $\delta$-neighbourhood of $I_1^\alpha$ is still contained in $I_2^\alpha$. Let $I_1$ be $\delta/2$-neighbourhood of $I_1^\alpha$.

A simple geometric observation of robustness type is that if $|\alpha-\theta|<\delta M^r/3$ then

\[
\Pi_\alpha \circ \varphi_{\underline{i}_r, \underline{j}_r}(I_1^\alpha) \subset \Pi_\theta \circ \varphi_{\underline{i}_r, \underline{j}_r}(I_1).
\]
Hence, Condition A($\alpha$) holds for all $\theta\in [\alpha-\delta M^r/3, \alpha+\delta M^r/3]$ for intervals $I_1$, $I_2$ and positive integer $r$.
\end{proof}

A natural corollary is that the whole range $\mathcal{D}$ can be presented as a countable union of closed intervals $J_i=[\alpha_i^-, \alpha_i^+]$ such that Condition A($\alpha$) holds for all $\alpha\in J_i$ with the same $I_1^i, I_2^i, r_i$. To prove Theorem \ref{thm:many} we only need to prove that for almost all $E$ and for any $i$, almost surely all the sets $\Pi_\alpha(E), \alpha\in J_i$ contain intervals (the horizontal and vertical directions follow from Falconer and Grimmett \cite{FG}).

\subsection{The proof}

We assume in this section that Condition A($\alpha$) holds with given $I_1, I_2$ and $r$ ($\alpha$ is fixed, so we suppress index $\alpha$). We will prove that there is a positive probability that $\Pi_\alpha(E) \supset I_1$.

For any $x\in \Delta_\alpha$, let us define a sequence of random variables

\[
V_n(x) = \sharp \{(\underline{i}_{nr}, \underline{j}_{nr})\in \mathcal{E}_{nr} \cap D_{nr}(x, I_1, \alpha)\}.
\]
For any $n$, let us define a finite set $X_n\subset I_1$ with the following properties:
\begin{itemize}
\item[i)] $X_n$ contains the endpoints of $I_1$,
\item[ii)] when we number the points of $X_n$ in increasing direction as $x_0,\ldots,x_N$ (with $x_0, x_N$ being the endpoints of $I_1$) then whenever $(\underline{i}_{nr}, \underline{j}_{nr})\in D_{nr}(x_i, I_1, \alpha)$, it will follow that for all $y\in[x_{i-1}, x_{i+1}]$, $(\underline{i}_{nr}, \underline{j}_{nr})\in D_{nr}(y, I_2, \alpha)$,
\item[iii)] $\sharp X_n \leq c M^{nr}$.
\end{itemize}
To have property ii) satisfied, it is enough to choose $X_n$ as points in regular distances $\delta M^{-nr}$ from each other, where $\delta$ is sufficiently small that $\delta$-neighbourhood of $I_1$ is still contained in $I_2$. So constructed $X_n$ will satisfy iii) as well.

We will prove that there is a positive probability that for all $n\in \mathbb N$, for all $x\in X_n$ we have $V_n(x) \geq (\frac 3 2)^n$. Note that that will imply the assertion: when all the points from $X_n$ will be contained in some $\Pi_\alpha\circ \varphi_{\underline{i}_{nr}, \underline{j}_{nr}}(I_1), (\underline{i}_{nr}, \underline{j}_{nr})\in \mathcal{E}_{nr}$, ii) will imply that whole $I_1$ will be contained in the union of corresponding $\Pi_\alpha\circ \varphi_{\underline{i}_{nr}, \underline{j}_{nr}}(I_2)$ and in particular in the union of corresponding $\Pi_\alpha\circ K_{\underline{i}_{nr}, \underline{j}_{nr}}$.

For $n=0$ the statement holds with probability 1. Assume that up to time $n$ it holds with probability $P_n$ and let us estimate the conditional probability with which it holds at time $n+1$, conditioned on the assumption it holds at time $n$. Let $x\in X_{n+1}$.

1. The point $x$ does not need to belong to $X_n$. However, even if it does not, it is contained in some $[x_i, x_{i+1}]$ for $x_i, x_{i+1}\in X_n$. As we assume that $V_n(x_i) \geq (3/2)^n$, we know that the number of pairs $(\underline{i}_{nr}, \underline{j}_{nr})\in \mathcal{E}_{nr} \cap D_{nr}(x_i, I_1, \alpha)$ is at least $(3/2)^n$. By part ii) of definition of $X_n$, all those $(\underline{i}_{nr}, \underline{j}_{nr})$ belong to $\mathcal{E}_{nr} \cap D_{nr}(x, I_2, \alpha)$ as well.

2. For each square $K_{\underline{i}_{nr}, \underline{j}_{nr}}, (\underline{i}_{nr}, \underline{j}_{nr})\in \mathcal{E}_{nr} \cap D_{nr}(x, I_2, \alpha)$ we want to calculate the number of its subsquares $K_{\underline{i}_{(n+1)r}, \underline{j}_{(n+1)r}}, (\underline{i}_{(n+1)r}, \underline{j}_{(n+1)r})\in \mathcal{E}_{(n+1)r} \cap D_{nr}(x, I_1, \alpha)$. This random number is given by

\[
G_\alpha^{(r)} \ind_{I_1}(\psi_{\underline{i}_{nr}, \underline{j}_{nr}}(x)).
\]
We do not know exactly the distribution of this random variable (it depends on $x$). But the possible values are obviously between 0 and $2M^{r}$ and the expected value is not smaller than

\[
F_\alpha^{r} \ind_{I_1}(\psi_{\underline{i}_{nr}, \underline{j}_{nr}}(x)) \geq 2
\]
(by Condition A($\alpha$) and using the fact that $\psi_{\underline{i}_{(n+1)r}, \underline{j}_{(n+1)r}}(x) \in I_2$).

3. Events that happen in different squares $K_{\underline{i}_{nr}, \underline{j}_{nr}}$ are jointly independent.

4. Hence, $V_{n+1}(x)$ is bounded from below by a sum of at least $(3/2)^n$ independent random variables, each with average $2$ and each bounded above and below by uniform constants. Hence, by Azuma-Hoeffding inequality \cite{Hoff} probability that this sum is strictly smaller than $(3/2)^{n+1}$ is not greater than $\gamma^{(3/2)^n}$ for some fixed $\gamma\in (0,1)$.

What we said implies that

\[
P\left(\forall_{x\in X_{n+1}} V_{n+1}(x)\geq (3/2)^{n+1} | \forall_{y\in X_n} V_n(y)\geq (3/2)^n\right) \geq \left(1-\gamma^{(3/2)^n}\right)^{cM^{(n+1)r}}.
\]

As the infinite sum $\sum_n cM^{(n+1)r}\gamma^{(3/2)^n}$ is convergent, we get

\[
P\left(\forall_n \forall_{x\in X_n} V_n(x)\geq (3/2)^n\right) >0.
\]
We are done.

\subsection{Examples}

Condition A($\alpha$) looks artificial, hence we should show some examples. The main goal of this subsection is to show that if all probabilities $p_{i,j}=p>M^{-1}$ then A($\alpha$) holds for all $\alpha\in\mathcal{D}$ (Proposition \ref{prop:alpha}), but we also mention some examples with different probabilities. Our main tool will be the following.

\begin{definition}
We say that the fractal percolation model satisfies {\bf Condition B($\alpha$)} if there exists a nonnegative continuous function $f:\Delta_\alpha\to \R$ such that $f$ is strictly positive except at the endpoints of $\Delta_\alpha$ and that
\begin{equation} \label{25}
F_\alpha f \geq (1+\varepsilon)f
\end{equation}
for some $\varepsilon>0$.
\end{definition}

First we prove
\begin{lemma}\label{29}
Assume that Condition B($\alpha$) holds for some $f$ and $\varepsilon >0$. Then we can choose nonempty closed intervals
\begin{equation*}\label{32}
I_1\subset {\rm int}I_2 \mbox{ and }I_2\subset {\rm int}\Delta,
\end{equation*}
such that for
\begin{equation*}\label{30}
g_{1}=f|_{I_{1}},\
g_{2}=f|_{I_{2}}
\end{equation*}
we have
\begin{equation}\label{31}
F_\alpha {g}_1(x)\geq\left(1+\frac{\varepsilon
}{2}\right)\cdot {g}_{2}(x)\mbox{ for } x\in
I_{2}.
\end{equation}
\end{lemma}
\begin{proof}


For a set
$H\subset \Delta_\alpha$, put $B_r(H)$ for the radius $r$ open
neighborhood of $H$ in $\Delta_\alpha$.
$$
B_r(H):=\left\{y\in \Delta_\alpha :\exists h\in H,\
|h-y|<r \right\}.
$$
Let $W\subset \Delta_\alpha $ be the $\Pi_\alpha$-projection of the mesh $1/M$ grid points in $K$:
$$
W=\left\{x\in \Delta :
\exists 0\leq i,j\leq M,\ x=\Pi_\alpha\left(\frac{i}{M},\frac{j}{M}\right)\right\}.
$$
We partition $W$ into the two endpoints of $\Delta$
(to be denoted $W_0$)
and $W_1:=W\setminus W_0$.
Fix $\eta >0$ which satisfies
\begin{equation*}\label{39}
\frac{\varepsilon }{2}\cdot   \min_{x\in B_{\eta/M}
(W_1)}f(x)>(M+1)^2\sup_x\left\{f(x):x\in B_{\eta /M}(W_0)\right\}.
\end{equation*}

and define two subintervals of $\Delta_\alpha $
\begin{equation*}\label{35}
 I_{1} :=\Delta_\alpha   \setminus B_\eta (W_0) \mbox{ and }
 I_2:=\Delta_\alpha   \setminus B_{\eta/M} (W_0).
\end{equation*}
Let
\begin{equation*}\label{34}
B=B_{\eta /M}(W) \mbox{ and }   B_i=B_{\eta /M}(W_i),i=0,1.
\end{equation*}
Fix an arbitrary  $x\in I_{2}$.
We  divide the proof of (\ref{31}) into two cases among which the first is obvious:
\begin{description}
\item[$\mathbf{x\in \Delta \setminus B}$]
Using the definition of $F_\alpha$ and then (\ref{25}) we obtain
$$
F_\alpha g_1(x)=F_\alpha f(x)\geq (1+\varepsilon )f(x)\geq
\left(1+\frac{\varepsilon }{2}\right)\cdot g_2(x).
$$
\item[$\mathbf{x\in B_1}$] By the definition of $F_\alpha $:
\begin{equation*}\label{36}
 F_\alpha g_1(x)\geq F_\alpha f-(M+1)^2\|f-\widetilde{g}_1\|_\infty,\  \forall x\in \Delta.
\end{equation*}
\end{description}
From this and from \eqref{25}, we obtain
$$
F_\alpha g_1(x)\geq \left(1+\frac{\varepsilon }{2}
\right)f(x)+\left(\frac{\varepsilon }{2}f(x)-
(M+1)^2\|f-\widetilde{g}_1\|_\infty\right)
$$
The definition of $\eta $ yields that the expression in the second bracket is positive. This implies that
$$
F_\alpha g_1(x)>
\left(1+\frac{\varepsilon }{2} \right)g_2(x)
\mbox{ for } x\in \Delta _2.
$$
\end{proof}

\begin{proposition} \label{prop:ba}
B($\alpha$) implies A($\alpha$).
\end{proposition}
\begin{proof}
Using the notation of Lemma \ref{29} we define $r$ as the smallest integer satisfying
\begin{equation*}\label{37}
 \left(1+\frac{\varepsilon }{2}\right)^r\geq
 2\cdot \frac{\max\limits_{x\in I_{1}}g_{1}(x)}
 {\min\limits_{x\in I_2}g_2(x)}
\end{equation*}
Then clearly,
$$
F_\alpha^r\ind_{I_1}(x)\geq 2\cdot
\ind_{I_2}(x) \mbox{  for all }  x\in I_{2}.
$$
\end{proof}

\begin{proposition} \label{prop:alpha}
If
\begin{equation*}\label{23}
 \forall i,j \qquad
 p_{ij}=p>\frac{1}{M}
\end{equation*}
then Condition A($\alpha$) is satisfied for all $\alpha\in\mathcal{D}$.
\end{proposition}
\begin{proof}
We will actually prove B($\alpha$).
Fix $\alpha \in \mathcal{D}$. For
an arbitrary $x\in \Delta^\alpha$ we define $f_\alpha (x):=|\ell^\alpha (x)\cap K|$ . It is straightforward that $f_\alpha$
satisfies (\ref{25}) with $\varepsilon =M\cdot p-1>0$.
\end{proof}

Let us now give some examples of percolations with not all probabilities equal and still satisfying Condition A($\alpha$). There is a large class of trivial examples given by the following lemma.

\begin{lemma}
If the percolation $\{p_{i,j}\}$ satisfies Condition ($\alpha$) and $p_{i,j}'\geq p_{i,j}$ for all $i,j$ then the percolation $\{p_{i,j}'\}$ satisfies Condition ($\alpha$) as well.
\end{lemma}

So, nontrivial examples should have at least some $p_{i,j}\leq M^{-1}$. A natural class of examples is motivated by the work of Dekking and Meester \cite{DM} and by the question of the anonymous referee.

\begin{lemma}\label{216}
Let $M=3$. Let $p_{1,1}=p_0$ and let all the other $p_{i,j}=p$. Then if
\[
p>\max\left(\frac 13, \frac {1-p_0} 2\right)
\]
then Condition A($\alpha$) is satisfied for all $\alpha\in \mathcal{D}$.
\end{lemma}
\begin{proof}
One can check that the Condition B($\alpha$) is satisfied for the same function $f_\alpha$ as in the proof of Proposition \ref{prop:alpha}.
\end{proof}

In the case $p_0=0$ we get the random Sierpi\'nski carpet and the Condition (A) is satisfied if $p>1/2$. Note that the bounds in Lemma \ref{216} are sharp: for $p_0\leq 1/3$ and $p\leq (1-p_0)/2$ the horizontal and vertical projections of $E$ almost surely contain no intervals, by Falconer and Grimmett \cite{FG}.

\section{Projections in many directions, proof of Theorem \ref{thm:many}}

We restrict ourself to one such range $J=[\alpha_-,\alpha_+]$. Let $I_1, I_2$ and $r$ are such that Condition A($\alpha$) holds for all $\alpha\in J$. Let $\delta$ be the Hausdorff distance between $I_1$ and $I_2$.


Another simple robustness-related geometric observation: assume $x,y\in \Delta_\alpha$ and the distance between them is at most $\delta M^{-nr}/3$. Assume $\alpha, \beta\in J$ and $|\alpha-\beta|\leq \delta M^{-nr}/3$. Assume that $(\underline{i}_{nr}, \underline{j}_{nr}) \in D_{nr}(x, I_1, \alpha)$. Then $(\underline{i}_{nr}, \underline{j}_{nr}) \in D_{nr}(y, I_2, \beta)$. We can write this as

\begin{equation} \label{g}
G_\beta^{(r)} \ind_{I_2}(\psi_{\underline{i}_{nr}, \underline{j}_{nr}}(y)) \geq G_\alpha^{(r)} \ind_{I_1}(\psi_{\underline{i}_{nr}, \underline{j}_{nr}}(x)).
\end{equation}

We are now starting the proof. Compare the proof of Theorem \ref{thm:ortho}.
Given $n$, let $X_n$ be a $\delta M^{-nr}/3$-dense finite subset of $I_1$ and let $Y_n$ be a $\delta M^{-nr}/3$-dense finite subset of $J$. We choose them in such a way that

\[
\sharp (X_n\times Y_n) \leq c M^{2nr}.
\]

For any $(x, \theta)\in I_1\times J$, let us define a sequence of random variables

\[
V_n(x, \theta) = \sharp\{(\underline{i}_{nr}, \underline{j}_{nr})\in \mathcal{E}_{nr}\cap D_{nr}(x, I_1, \theta)\}.
\]
We will prove that with positive probability $V_n(x, \theta)\geq (3/2)^n$ for all $n, x, \theta$, estimating inductively the probability that this event holds up to time $(n+1)$ conditioned on the assumption that it holds at time $n$. For $n=0$ this event holds with probability 1. Let us start the inductive step.

1. Given $(y, \kappa) \in X_{n+1}\times Y_{n+1}$, let $Z(y, \kappa)$ be the set of points  from $I_1\times J$ such that
$x$ is $\delta M^{-(n+1)r}/3$-close to $y$ and $\theta$ is $\delta M^{-(n+1)r}/3$-close to $\kappa$. The sets $Z(y,\kappa)$ cover $I_1\times J$.

By the inductive assumption, $V_n(y, \kappa) \geq (3/2)^n$. Hence, we know that there are at least $(3/2)^n$ pairs $(\underline{i}_{nr}, \underline{j}_{nr})\in \mathcal{E}_{nr} \cap D_{nr}(y, I_2, \kappa)$.

2. For each square $K_{\underline{i}_{nr}, \underline{j}_{nr}}$ such that $(\underline{i}_{nr}, \underline{j}_{nr})\in \mathcal{E}_{nr} \cap D_{nr}(x, I_1, \theta)$, we want to calculate the number of its subsquares $K_{\underline{i}_{(n+1)r}, \underline{j}_{(n+1)r}}$ such that $(\underline{i}_{(n+1)r}, \underline{j}_{(n+1)r})\in \mathcal{E}_{(n+1)r} \cap D_{(n+1)r}(x, I_2, \theta)$. This random number is given by $G_\theta^{(r)} \ind_{I_2}(\psi_{\underline{i}_{nr}, \underline{j}_{nr}}(x))$ and by \eqref{g}

\begin{equation} \label{g2}
G_\theta^{(r)} \ind_{I_2}(\psi_{\underline{i}_{nr}, \underline{j}_{nr}}(x)) \geq G_\kappa^{(r)} \ind_{I_1}(\psi_{\underline{i}_{nr}, \underline{j}_{nr}}(y)).
\end{equation}
Like before, this random variable is bounded (independently of $n$) and its expected value is at least 2. Moreover, those random variables coming from different $(\underline{i}_{(n+1)r}, \underline{j}_{(n+1)r})$ are independent.

3. An important note: the bound in equation \eqref{g2} works for all $(x,\theta) \in Z(y, \kappa)$. That means that we only need to check the behaviour of this random variable for finitely many pairs $(y, \kappa)$ to prove the inductive step at all $(x, \theta)$.

4. By Azuma-Hoeffding inequality the conditional probability that

\[
\sum_{\underline{i}_{nr}, \underline{j}_{nr} \in \mathcal{E}_{nr} \cap D_{nr}(y, I_2, \kappa)} G_\kappa^{(r)} \ind_{I_1}(\psi_{\underline{i}_{nr}, \underline{j}_{nr}}(y)) < (3/2)^{n+1}
\]

 conditioned on $V_n(y, \kappa) \geq (3/2)^n$ is not greater than $\gamma^{(3/2)^n}$ for some fixed $\gamma\in (0,1)$. As the number of possible pairs $(y,\kappa)$ is at most $cM^{2nr}$, which  is increasing only exponentially fast, we are done.

\section{Nonlinear projections, proof of Theorem \ref{thm:nonlin}}

\subsection{Almost linear projections}

Let us consider carefully what are the real assumptions of the proof of Theorem \ref{thm:many}. Consider a family of projections $S_t: K\to \Delta$ parametrized by $t\in T$. A convenient way will be to write

\begin{equation} \label{eqn:st}
S_t(x)=\Pi_{\alpha_t(x)}(x)
\end{equation}
for all $x\in K$. What assumptions about $\alpha_t$ we would need for the proof from previous section to work?

We want to use Condition A. So, our first necessary assumption is that for some range $J$ in which Condition A holds (for some fixed $I_1, I_2, r$), $\alpha_t(x)\in J$ for all $t,x$. Let $\delta$ be, like before, the Hausdorff distance between $I_1$ and $I_2$.

We want also the following robustness property. For any $n$ we want to be able to divide $I_1\times T$ into a finite family of subsets $\{X_i \times Z_j\}$ and in each $X_i \times Z_j$ we want to choose a special pair $(x_i,t_j) \in X_i \times Z_j$ such that for any $(x,t)\in X_i\times Z_j$ and for any $\underline{i}_{nr}, \underline{j}_{nr}$,

\begin{equation*} \label{eee}
x_i\in S_{t_j} \circ \varphi_{\underline{i}_{nr}, \underline{j}_{nr}}(I_1) \implies x\in \Pi_{\alpha_t(X_{\underline{i}_{nr}, \underline{j}_{nr}})} \circ \varphi_{\underline{i}_{nr}, \underline{j}_{nr}}(I_2),
\end{equation*}
where $X_{\underline{i}_{nr}, \underline{j}_{nr}}$ is the center of $K_{\underline{i}_{nr}, \underline{j}_{nr}}$. This will let us proceed with the inductive part of the argument.

Finally, we need the size of the family $\{Z_i\}$ to grow only exponentially fast with $n$, so that we can apply the large deviation argument and the resulting infinite product is convergent.

\begin{definition} \label{def:alfop}
We say that a family $\{S_t\}_{t\in T}:K\to \Delta$ is an {\bf almost linear family of projections} if the following properties are satisfied. We use notation from \eqref{eqn:st}. We set $J\subset \mathcal{D}$ as the range of angles for which Condition A($\alpha$) is satisfied with the same $I_1, I_2, r$. We denote by $\delta$ the Hausdorff distance between $I_1$ and $I_2$.

\begin{itemize}
\item[i)] $\alpha_t(x)\in J$ for all $t\in T$ and $x\in K$. In particular, $\alpha_t(x)$ is contained in one of two components of $\mathcal{D}$.
\item[ii)] $\alpha_t(x)$ is a Lipschitz function of $x$, with the Lipschitz constant not greater than $\delta/3$. This guarantees in particular that $S_t(K_{\ien, \jen})$ is an interval.
\item[iii)] For any $n$ we can divide $T$ into subsets $Z_i^{(n)}$ such that whenever $t,s\in Z_i^{(n)}$ and $x,y\in K_{\ien, \jen}$, we have
    \[
    |\alpha_t(x)-\alpha_s(y)| \leq \delta M^{-n}/3.
    \]
Moreover, we can do that in such a way that $\sharp \{Z_i^{(n)}\}$ grows only exponentially fast with $n$.
\end{itemize}
\end{definition}

Then the proof of Theorem \ref{thm:many} easily yields the following.

\begin{theorem} \label{thm:general}
Let $\{S_t\}_{t\in T}$ be an almost linear family of projections. Then for almost all nonempty realizations $E$ of the percolation fractal, $S_t(E)$ contains an interval for all $t\in T$.
\end{theorem}
\begin{proof}

We denote by $V_n(x,t)$ the number of pairs $(\underline{i}_{nr}, \underline{j}_{nr})\in \mathcal{E}_{nr}$ for which $x\in S_t\circ \varphi_{\underline{i}_{nr}, \underline{j}_{nr}}(I_2)$. We want to prove inductively that (with positive probability) $V_n(x,t) \geq (3/2)^n$ for all $x\in I_1, t\in T$. The statement is obvious for $n=0$. The inductive step is as follows.

1. We choose in $I_1$ a $\delta M^{-(n+1)r}/3$-dense finite subset $X_{n+1}$. We can cover $I_1\times T$ with sets $B_{\delta M^{-(n+1)r}/3}(x_i) \times Z_j^{((n+1)r)}$, $x_i\in X_{n+1}$. The inductive assumption says that for any $(y,s)$ there are at least $(3/2)^n$ pairs $(\underline{i}_{nr}, \underline{j}_{nr})\in \mathcal{E}_{nr}$ such that $y\in S_s\circ\varphi_{\underline{i}_{nr}, \underline{j}_{nr}}(I_2)$.

2. For each $K_{\underline{i}_{nr}, \underline{j}_{nr}}, (\underline{i}_{nr}, \underline{j}_{nr})$ as above, we want to estimate from below the number of its subsquares $K_{\underline{i}_{(n+1)r}, \underline{j}_{(n+1)r}}$ such that $y\in S_s\circ\varphi_{\underline{i}_{(n+1)r}, \underline{j}_{(n+1)r}}(I_2)$. For all $(y,s)\in B_{\delta M^{-(n+1)r}/3}(x_i) \times Z_j^{((n+1)r)}$ this random variable can be uniformly estimated from below by

\[
G_{\alpha_t(X_{\underline{i}_{nr}, \underline{j}_{nr}})}^{(r)} \ind_{I_1}(\psi_{\underline{i}_{nr}, \underline{j}_{nr}}(x_i),
\]
where $t\in Z_j^{((n+1)r)}$ is arbitrary.

3. As we approximate the almost linear projection by a linear one, we can apply Condition A($\alpha_t(X_{\underline{i}_{nr}, \underline{j}_{nr}})$).

4. As the number of sets $B_{\delta M^{-(n+1)r}/3}(x_i) \times Z_j^{((n+1)r)}$ grows only exponentially fast with $n$, we finish the proof using Azuma-Hoeffding inequality, like before.
\end{proof}

\subsection{Radial and co-radial projections}

Families of radial and co-radial projections are not in general almost linear families of projections. However, as explained in section 3, we only need to consider radial/co-radial projections with center in uniformly nonhorizontal, nonvertical direction and arbitrarily big distance from $K$. If we fix any nonhorizontal and nonvertical direction and consider only centers in sufficiently large distance, the resulting family of radial projections and family of co-radial projections will satisfy conditions ii), iii) of Definition \ref{def:alfop}. To have the condition i) satisfied as well, we only need to subdivide the family. Hence, Theorem \ref{thm:nonlin} follows from Theorem \ref{thm:general}.


\end{document}